\documentclass[conference]{IEEEtran}
\IEEEoverridecommandlockouts
\usepackage[top=1in,left=0.75in,right=0.75in,bottom=0.75in]{geometry}
\usepackage{cite}
\usepackage{amsmath,amssymb,amsfonts,amsthm}
\usepackage{graphicx}
\usepackage{xcolor}
\usepackage{url}
\usepackage{hyperref}

\newtheorem{definition}{Definition}
\newtheorem{proposition}{Proposition}
\newtheorem{theorem}{Theorem}
\newtheorem{lemma}{Lemma}
\newtheorem{remark}{Remark}

\begin{document}
\title{Higher-Order Mean-Field \\
Control Barrier Functions
\thanks{The authors contributed equally.}}
\author{\IEEEauthorblockN{Levon Nurbekyan}
\IEEEauthorblockA{\textit{Department of Mathematics}\\
\textit{Emory University}\\
Atlanta, USA\\
lnurbek@emory.edu}
\and
\IEEEauthorblockN{Samy Wu Fung}
\IEEEauthorblockA{\textit{Department of Applied Mathematics and Statistics}\\
\textit{Colorado School of Mines}\\
Golden, USA\\
swufung@mines.edu}}
\maketitle

\begin{abstract}
Mean-field control barrier functions (MF-CBFs) enforce swarm safety through constraints on the agents' distribution. Previous formulations, including applications to stochastic coverage and shepherding, use first-order differential inequalities. These cannot directly enforce constraints whose barrier functions yield control-free first derivatives. Thus, we develop the theory of higher-order MF-CBFs, which ensures the positivity of the barrier functionals through higher-order differential inequalities, as in their finite-dimensional analogs. More specifically, we introduce the notion of relative degree for mean-field safety functionals and develop an analytic framework for computing their higher-order derivatives. Moreover, for practical examples of cross-correlation and self-correlation functionals, we show that repeated differentiation along mean-field dynamics is structure-preserving and reduces to kernel recursions. As an illustration, we apply higher-order MF-CBFs to double-integrator swarm tracking and avoidance problems, where position-only constraints lead to MF-CBFs of relative degree two.
\end{abstract}


\section{Introduction}\label{sec:introduction}

Control barrier functions (CBFs) provide an effective method for ensuring safety in control systems~\cite{ames2016control,ames2019control}. The fundamental idea is to introduce a barrier function whose nonnegativity encodes safety and preserve it along the dynamics. A differential inequality on this function defines a set of safe controls at each state and time. The controller then checks the nominal control and, if it is unsafe, projects it onto the safe set.

The CBF approach accommodates stabilizing, optimal, idle, or other nominal controls, although their global behavior after filtering must be analyzed on a case-by-case basis. Beyond safety and flexibility, CBFs also localize the projection in space and time; the correction only requires the current state and time, without solving any global-in-time or global-in-space optimization problem. For control-affine systems with quadratic projection objectives and affine input constraints, this amounts to a quadratic program~\cite{ames2016control}.

For large swarms of identical agents, the objectives often concern aggregate behavior, such as coverage, tracking, or occupying prescribed regions~\cite{sinigaglia2022density,elamvazhuthi2023density,elamvazhuthi2024denoising,lasry2007mean,ruthotto2020machine,lin2021alternating,Liu2021splitting,vidal2024kernel,maffettone2025leader-follower,niu2026sparse}. In such settings, it is natural to treat the swarm as a whole and model it as a probability distribution in state space. Its control and dynamics are then considered in the space of probability measures. Borrowing terminology from statistical physics, these are called mean-field control problems~\cite{lasry2007mean,fornasier2014mean}.

The MF-CBF framework was introduced in~\cite{FungNurbekyan2025mfcbf} and extended to stochastic agents for coverage and shepherding in~\cite{tomaselli2026mfcbf}. These works use first-order differential inequalities to preserve the nonnegativity of a safety functional. However, control may not appear in its first derivative along the dynamics. For example, acceleration first affects a position constraint through the second derivative. As in finite-dimensional CBF theory, such constraints require higher-order differential conditions~\cite{krstic2006nonovershooting,nguyen2016expCBF,xiao2022hoCBF,xiao2019hoCBF}.

In this work, we introduce relative degree for mean-field safety functionals with respect to the individual-agent dynamics and develop the calculus needed to apply the higher-order CBF mechanism. To the best of our knowledge, higher-order MF-CBFs have not previously been considered. More specifically, our contributions are as follows.
\begin{itemize}
    \item A notion of relative degree and an analytic framework for computing higher derivatives of functionals along mean-field dynamics and obtaining the resulting safety conditions.
    \item Analysis of cross-correlation and self-correlation functionals and proving that they retain their respective structures.
    \item Explicit safety constraints for double integrators with position-only correlation functionals of relative degree two.
    \item Swarm tracking and obstacle avoidance experiments with $N=100$ double-integrator agents, using particle controls and interpolated feedback fields.
\end{itemize}

The rest of the paper is organized as follows. Section~\ref{subsec:related} discusses related work, whereas Section~\ref{sec:preliminaries} reviews the basic CBF mechanism. Section~\ref{sec:derivatives} develops the theory of higher-order MF-CBFs, and Section~\ref{sec:kernels} provides the analysis of correlation barriers. Finally we perform numerical experiments in Section~\ref{sec:experiments} followed by brief conlusions in Section~\ref{sec:conclusions}.

\subsection{Related work}\label{subsec:related}

In~\cite{FungNurbekyan2025mfcbf}, the authors introduced the control barrier framework in the mean-field setting. Similar ideas were developed in~\cite{yi2025constrainedInference} for constrained variational inference, where MF-CBFs were used to adjust the vector field in the gradient flow of the KL divergence so that the probability distributions along the flow satisfy suitable (safety) constraints. In~\cite{gao2026banachcontrolbarrierfunctions}, the authors consider a generalization of MF-CBFs, where the control barrier functionals are extended to control barrier operators, and scalar positivity constraints are replaced by positivity constraints in a Banach space $L^1(\mathbb{R}^d)$ once it is equipped with a suitable partial order (e.g., $f\succeq g$ iff $f(x)\geq g(x)$ for a.e. $x\in \mathbb{R}^d$). This approach allows for pointwise constraints on the density as opposed to mean-field (aggregate) constraints in~\cite{FungNurbekyan2025mfcbf,yi2025constrainedInference} (although one can engineer mean-field constraints that yield pointwise density constraints~\cite{yi2025constrainedInference,Liu2021splitting}). In~\cite{tomaselli2026mfcbf}, the authors extend the MF-CBF framework to systems where individual agent dynamics are stochastic and applied it to coverage and shepherding control problems. Similar ideas are employed in~\cite{niu2026sparse}.

MF-CBFs are control barrier functionals since the state of the system is infinite-dimensional; in our case, the state of the system is the distribution of the agents in the state space. Infinite-dimensional control barriers have been explored before in other contexts as well. In~\cite{kiss2023cbFunctionals}, the authors used control barrier functionals to ensure safety in time delay systems. The manuscript~\cite{koga2022safePDE} employed (higher-order) control barrier functionals for the control of the Stefan PDE. Similar ideas were used in~\cite{ahmadi2017safety} for studying the safety behavior of (uncontrolled) PDE.

Finally, (finite-dimensional) higher-order CBFs were introduced in~\cite{xiao2019hoCBF,nguyen2016expCBF,krstic2006nonovershooting}.

\section{Preliminaries: the basic mechanism underlying CBFs}\label{sec:preliminaries}

The basic mechanism of the CBFs addresses the following question: given a smooth function $h_0(t)$, what differential conditions guarantee that $h_0(t)\geq0$ for all $t\geq0$? Recall that a continuous function $\alpha:\mathbb R\to\mathbb R$ belongs to extended class-$\mathcal K_\infty$ if it is strictly increasing, $\alpha(0)=0$, and $\lim_{s\to \pm \infty} \alpha(s)=\pm \infty$. The following result is the mathematical underpinning of the CBF approach.

\begin{lemma}\label{lem:scalar_positivity}
Let $h_0:[0,\infty)\to\mathbb R$ be smooth, and let $\alpha_1,\ldots,\alpha_r$ be smooth extended class-$\mathcal K_\infty$ functions. Define
\begin{equation}\label{eq:prelim_recursion}
    h_k(t)=\dot h_{k-1}(t)+\alpha_k(h_{k-1}(t)),\qquad 1\leq k\leq r.
\end{equation}
If $h_k(0)\geq 0$ for $0\leq k<r$, and
\begin{equation}\label{eq:prelim_positivity_conditions}
\begin{aligned}
    % h_k(0)&\geq0,&&,\\
    h_r(t)&\geq0,&&t\geq0,
\end{aligned}
\end{equation}
then $h_k(t)\geq0$ for every $0\leq k<r$ and $t\geq0$. In particular, $h_0(t)\geq0$ for all $t\geq0$.
\end{lemma}

The key idea is to propagate nonnegativity down the recursion. A scalar differential inequality $\dot h+\alpha(h)\geq0$ preserves $h\geq0$ when $h(0)\geq0$. Starting with $h_r\geq0$, this observation first gives $h_{r-1}\geq0$, then $h_{r-2}\geq0$, and eventually $h_0\geq0$. See~\cite[Lemma 1 \& Theorem 3]{xiao2022hoCBF} and references therein for more details.

Now consider a system with control-affine dynamics
\begin{equation}\label{eq:prelim_control_dynamics}
    \dot x=f(t,x)+g(t,x)u(t),\qquad x\in\mathbb R^d,
\end{equation}
where $f(t,x)\in\mathbb R^d$ and $g(t,x)\in\mathbb R^{d\times m}$ are defined for $(t,x)\in[0,\infty)\times\mathbb R^d$, and $u(t)\in U\subseteq \mathbb R^m$ is the control.

Furthermore, let $b(t,x)$ be a smooth (barrier) function whose nonnegativity describes safety. The goal is to choose $u(t)$ so that the system is safe all positive times; that is,
\begin{equation}\label{eq:prelim_h0}
    h_0(t)=b(t,x(t))\geq0,\qquad t\geq0.
\end{equation}
To this end, one applies Lemma~\ref{lem:scalar_positivity}, defining $h_k=\dot h_{k-1}+\alpha_k(h_{k-1})$ and choosing the control to enforce $h_r\geq0$, provided $h_k(0)\geq0$ for $0\leq k<r$.

The appropriate choice of $r$ depends on the first time the control appears in the sequence $(h_k)$ in~\eqref{eq:prelim_recursion} or, as Lemma~\ref{lem:scalar_expansion} demonstrates, the first time that it appears in $h_0^{(k)}$. We have that
\begin{equation}\label{eq:prelim_first_derivative}
\begin{aligned}
    \dot h_0(t)={}&\partial_t b(t,x)+D_x b(t,x)[f(t,x)+g(t,x)u(t)],
\end{aligned}
\end{equation}
where $D_x$ denotes the (row) Jacobian operator (as opposed to $\nabla_x$ which denoted the (column) gradient operator). We also omit the dependence of $x$ on $t$ for brevity. If $D_xb\,g$ is nonzero, $h_1\geq0$ is an affine constraint on $u$. If this coefficient vanishes identically, we differentiate further. When the control first appears in the $r$-th derivative of $h_0$, we call $r$ the relative degree of $b$ with respect to the dynamics~\eqref{eq:prelim_control_dynamics}; see~\cite[Definition~6]{xiao2022hoCBF}. The lower derivatives depend only on $(t,x(t))$, while the $r$-th derivative is affine in $u(t)$. Thus the higher-order condition $h_r\geq0$ can be used to select the control, subject to feasibility along the resulting evolution.

Thus, the practical aspect of applying CBFs is the computation of $(h_k)$ and identifying their dependence on the control. Xiao and Belta give a Lie-derivative expansion in~\cite[eqs.~(13)--(14)]{xiao2022hoCBF}. We revisit this computation from a slightly different angle: we first express each $h_k$ in~\eqref{eq:prelim_recursion} in terms of $h_0,\cdots,h_0^{(k)}$, without invoking the state dynamics. We then use the dynamics to compute these derivatives.

\begin{lemma}\label{lem:scalar_expansion}
Let $h_0,\alpha_1,\ldots,\alpha_r$ be smooth, and define $h_k$ by~\eqref{eq:prelim_recursion}. Then, for $0\leq k\leq r$,
\begin{equation}\label{eq:prelim_expansion}
    h_k(t)=h_0^{(k)}(t)+B_k\big(h_0(t),\ldots,h_0^{(k-1)}(t)\big),
\end{equation}
where $h_0^{(0)}=h_0$, $B_0=0$, and
\begin{equation}\label{eq:prelim_B_recursion}
\begin{aligned}
    &B_k(y_0,\ldots,y_{k-1})\quad=\sum_{i=1}^{k-1}\partial_{y_{i-1}}B_{k-1}(y_0,\ldots,y_{k-2})\,y_i\\
    &+\alpha_k\big(y_{k-1}+B_{k-1}(y_0,\ldots,y_{k-2})\big),
\end{aligned}
\end{equation}
for $1\leq k\leq r$, where empty sums are discarded.
\end{lemma}
\begin{proof}
The cases $k=0$ follow from $B_0=0$. Suppose~\eqref{eq:prelim_expansion} holds for some $0\leq k<r$. Using~\eqref{eq:prelim_recursion} and differentiating~\eqref{eq:prelim_expansion} yields
\begin{equation*}
\begin{aligned}
    h_{k+1}
    ={}&h_0^{(k+1)}+\sum_{i=1}^{k}\partial_{y_{i-1}}B_k\big(h_0,\ldots,h_0^{(k-1)}\big)\,h_0^{(i)}\\
       &+\alpha_{k+1}\big(h_0^{(k)}+B_k\big(h_0,\ldots,h_0^{(k-1)}\big)\big)\\
    ={}&h_0^{(k+1)}+B_{k+1}\big(h_0,\ldots,h_0^{(k)}\big),
\end{aligned}
\end{equation*}
where the last equality follows from~\eqref{eq:prelim_B_recursion}. This completes the induction.
\end{proof}

The implication of identity~\eqref{eq:prelim_expansion} is that the first time the control appears in $(h_k)$ is precisely the first time the control appears in $(h_0^{(k)})$. Morever, functions $(B_k)$ depend only on the functions $(\alpha_k)$ and can be precomputed irrespective of the state dynamics or $h_0$. Thus, the computation of $(h_k)$ essentially reduces to the computation of $(h_0^{(k)})$.

In what follows, we develop the analytic framework for computing these higher order derivatives for $h_0(t)=\mathcal F(t,\rho_t)$ and applying the higher-order CBF principle in Lemma~\ref{lem:scalar_positivity}, where $\rho_t$ is the swarm density at time $t$, and $\mathcal{F}$ is the functional that encodes the safety of the swarm. 

\section{Higher-order mean-field control barriers}\label{sec:derivatives}\label{sec:safety}

Consider a swarm whose individual agents follow
\begin{equation}\label{eq:dynamics}
    \dot x=f(t,x)+g(t,x)u(t,x),\qquad x\in\mathbb R^d,
\end{equation}
where $f$ and $g$ are as in~\eqref{eq:prelim_control_dynamics}, and $u(t,x)\in U\subseteq\mathbb R^m$ is a feedback control. The feedback control assumption is necessary for deriving the mean-field dynamics.

More specifically, the probability distribution of the swarm evolves according to the continuity equation
\begin{equation}\label{eq:CE}
    \partial_t\rho_t+\nabla\cdot\big(\rho_t(f+gu)\big)=0.
\end{equation}

In the mean-field setting, safety is encoded by a functional $\mathcal F(t,\rho)$, and our objective is to choose the control so that
\begin{equation*}
    h_0(t)=\mathcal F(t,\rho_t)\geq0,\qquad t\geq0.
\end{equation*}
By Lemma~\ref{lem:scalar_expansion}, the first task for applying the CBF mechanism of Section~\ref{sec:preliminaries} is to compute the higher-order derivatives of $h_0$ along~\eqref{eq:CE}.

\subsection{Differentiation along the density dynamics}

Throughout the paper, we assume sufficient smoothness and decay to justify the differentiations and integrations by parts. The first variation of $\mathcal F$ is defined by
\begin{equation*}
    \left.\frac{d}{dh}\mathcal F(t,\rho+h\eta)\right|_{h=0}
    =\int_{\mathbb R^d}\delta_\rho\mathcal F(t,\rho)(x)\eta(x)\,dx
\end{equation*}
for admissible perturbations $\eta$ with $\int_{\mathbb R^d}\eta\,dx=0$. Furthermore, we define
\begin{equation}\label{eq:operators}
\begin{aligned}
    &D^W_\rho\mathcal F(t,\rho)(x)=D_x[\delta_\rho\mathcal F(t,\rho)(x)],\\
    &\mathcal L_f\mathcal F(t,\rho)=\partial_t\mathcal F(t,\rho)+\int_{\mathbb R^d} D^W_\rho\mathcal F(t,\rho)(x)f(t,x)\rho(x)\,dx.
\end{aligned}
\end{equation}
Thus $\mathcal L_f$ differentiates a functional along the \emph{uncontrolled} density dynamics, including its explicit time dependence. The chain rule yields
\begin{equation*}
    \frac{d\mathcal F(t,\rho_t)}{dt}
    =\partial_t\mathcal F(t,\rho_t)
    +\int_{\mathbb R^d}\delta_\rho\mathcal F(t,\rho_t)(x)\partial_t\rho_t(x)\,dx.
\end{equation*}
Substituting~\eqref{eq:CE} and integrating by parts yields
\begin{equation}\label{eq:chain_rule}
\begin{aligned}
    \frac{d\mathcal F(t,\rho_t)}{dt}
    ={}&\mathcal L_f\mathcal F(t,\rho_t)+\int_{\mathbb R^d} D^W_\rho\mathcal F(t,\rho_t)\,g\, u\,\rho_t\,dx,
\end{aligned}
\end{equation}
where we omit the dependence of integrands on $t,x$ for brevity.
This identity separates the drift contribution from the term involving the control. If the latter vanishes identically, we apply the same identity to $\mathcal L_f\mathcal F$ and continue until the control appears. As usual, we set $\mathcal{L}_f^0 \mathcal{F}=\mathcal{F}$.

\subsection{Relative degree and higher-order derivatives}
\begin{definition}\label{def:relative_degree}
The functional $\mathcal F$ has relative degree $r\geq1$ with respect to~\eqref{eq:CE} if
\begin{equation}\label{eq:relative_degree}
    D^W_\rho[\mathcal L_f^k\mathcal F](t,\rho)(x)g(t,x)\rho(x)=0
\end{equation}
for all $(t,\rho,x)$ and $0\leq k\leq r-2$, whereas the same expression with $k=r-1$ is nonzero for some $(t,\rho,x)$.
\end{definition}

The relative degree identifies the derivative in which the control can enter but it does not guarantee that an admissible control satisfying $h_r\geq0$ exists at every density encountered along the evolution. The latter question must be analyzed separately on a case by case basis.

\begin{proposition}\label{prop:derivatives}
Suppose that $\mathcal F$ has relative degree $r$ and $\rho_t$ solves~\eqref{eq:CE}. Let $h_0(t)=\mathcal F(t,\rho_t)$, and define $h_k$ by~\eqref{eq:prelim_recursion} for smooth extended class-$\mathcal K_\infty$ functions $\alpha_1,\ldots,\alpha_r$. Then
\begin{equation}\label{eq:derivatives}
\begin{aligned}
    h_0^{(k)}(t)&=\mathcal L_f^k\mathcal F(t,\rho_t),\qquad 0\leq k<r,\\
    h_0^{(r)}(t)&=\mathcal L_f^r\mathcal F(t,\rho_t)+\int_{\mathbb R^d}D^W_\rho[\mathcal L_f^{r-1}\mathcal F](t,\rho_t)\,g\, u\,\rho_t\,dx.
\end{aligned}
\end{equation}
Consequently, with $(B_k)$ as in Lemma~\ref{lem:scalar_expansion}, the auxiliary functions satisfy, for $0\leq k<r$,
\begin{equation}\label{eq:fk_density}
\begin{aligned}
    h_k(t)&=\mathcal L_f^k\mathcal F+B_k\big(\mathcal F,\ldots,\mathcal L_f^{k-1}\mathcal F\big),\\
    h_r(t)&=\mathcal L_f^r\mathcal F+B_r\big(\mathcal F,\ldots,\mathcal L_f^{r-1}\mathcal F\big)\\
    &+\int_{\mathbb R^d}D^W_\rho[\mathcal L_f^{r-1}\mathcal F](t,\rho_t)\,g\, u\,\rho_t\,dx,
\end{aligned}
\end{equation}
where all $\mathcal{L}_f^k\mathcal F$ are evaluated at $(t,\rho_t)$.
\end{proposition}
\begin{proof}
Equalities~\eqref{eq:derivatives} follow from a successive application of~\eqref{eq:chain_rule} to $\mathcal F,\mathcal L_f\mathcal F,\ldots,\mathcal L_f^{r-1}\mathcal F$. Definition~\ref{def:relative_degree} makes the control term vanish in the first $r-1$ differentiations, while the $r$-th differentiation retains the control integral. Substituting~\eqref{eq:derivatives} into~\eqref{eq:prelim_expansion} yields~\eqref{eq:fk_density}.
\end{proof}

As in the finite-dimensional case, $h_k$ depends only on $(t,\rho_t)$ for $k<r$, while $h_r$ is affine in the control, even for nonlinear $(\alpha_k)$.

\subsection{Safe controls and preservation of safety}
Let $\mathcal U$ denote the admissible control fields $u:\mathbb R^d\to U$ for which the integrals above are defined. Define
\begin{equation}\label{eq:control_general}
\begin{aligned}
    &\mathcal K_{\mathrm{CBF}}(t,\rho)=\big\{u\in\mathcal U:\;\mathcal L_f^r\mathcal F(t,\rho)\\
    &\quad+B_r\big(\mathcal F(t,\rho),\ldots,\mathcal L_f^{r-1}\mathcal F(t,\rho)\big)\\
    &\quad+\int_{\mathbb R^d}D^W_\rho[\mathcal L_f^{r-1}\mathcal F](t,\rho)\,g\, u\,\rho\,dx\geq0\big\}.
\end{aligned}
\end{equation}
Here $u=u(x)$ denotes a control field at a fixed time.
Membership in this set is precisely the condition $h_r\geq0$. It imposes one affine inequality on the control field, in addition to any admissibility restrictions such as $u(x)\in U$.

\begin{theorem}\label{thm:safety}
Let $\mathcal F$, $(\alpha_k)$, and $(B_k)$ be as in Proposition~\ref{prop:derivatives}. Assume that $\rho_t$ evolves according to~\eqref{eq:CE} for $t\geq0$, with initial density $\rho_0$ satisfying
\begin{equation}\label{eq:initial_general}
\begin{aligned}
    \mathcal L_f^k\mathcal F(0,\rho_0)+B_k\big(\mathcal F(0,\rho_0),\ldots,\mathcal L_f^{k-1}\mathcal F(0,\rho_0)\big)\geq0,
\end{aligned}
\end{equation}
for $0\leq k<r$, and that the control satisfies
\begin{equation}\label{eq:safe_control_inclusion}
    u(t,\cdot)\in\mathcal K_{\mathrm{CBF}}(t,\rho_t),\qquad t\geq0,
\end{equation}
Then $\mathcal F(t,\rho_t)\geq0$ for all $t\geq0$.
\end{theorem}
\begin{proof}
Proposition~\ref{prop:derivatives} identifies~\eqref{eq:initial_general} with $h_k(0)\geq0$ for $k<r$, and~\eqref{eq:safe_control_inclusion} with $h_r(t)\geq0$. Lemma~\ref{lem:scalar_positivity} then yields the conclusion.
\end{proof}
\begin{remark}
We assume that $f$, $g$, $u$, and $U$ are such that the trajectories of~\eqref{eq:dynamics} do not blow up in finite time and sufficiently fast decay of $\rho_0$ at infinity is preserved by $\rho_t$ for every $t\geq0$. These properties hold, for example, when the velocity field $f+gu$ is globally Lipschitz in space and locally uniformly in time; see~\cite[Section 8.1]{ambrosio08} for more details.
\end{remark}

\section{Correlation barriers for swarm control}\label{sec:kernels}

Theorem~\ref{thm:safety} provides the safety mechanism for the swarm; we need to ensure that the distributed control $u(t,\cdot)$ is in the set of safe controls $\mathcal K_{\mathrm{CBF}}(t,\rho_t)$ at all times $t\geq 0$.

Thus, we should be able to compute $\mathcal K_{\mathrm{CBF}}(t,\rho)$ for obtaining a practical algorithm. In principle, the drift derivatives $\mathcal L_f^j\mathcal F$ might lead to complex expressions due to higher-order differentiations in the space of probability measures. Here, we show that for two practical examples of the cross-correlation and self-correlation functionals, these differentiations are structure preserving. More specifically, if $\mathcal{F}$ is a cross-correlation or a self-correlation, then $\mathcal L_f^j\mathcal F$ remain cross-correlation and self-correlation, respectively. Moreover, differentiations reduce to kernel recursions. 

\subsection{Correlation functionals and kernel recursions}

Assume that $\rho_t^*$ is a reference density that represents either a tracking target or an obstacle and satisfies
\begin{equation}\label{eq:reference_CE}
    \partial_t\rho_t^*(t,y)+\nabla_y\cdot(\rho_t^*(t,y)f^*(t,y))=0,
\end{equation}
with a reference vector field $f^*$; a stationary $\rho^*_t$ corresponds to $f^*=0$. For smooth kernels $k$ and $h$, with $h(t,x,y)=h(t,y,x)$, we define
\begin{equation}\label{eq:cross}
    R(t,\rho)=\int_{\mathbb R^{2d}}k(t,x,y)\rho(x)\rho_t^*(y)\,dx\,dy
\end{equation}
and
\begin{equation}\label{eq:self}
    S(t,\rho)=\int_{\mathbb R^{2d}}h(t,x,y)\rho(x)\rho(y)\,dx\,dy.
\end{equation}
We use the barriers
\begin{equation}\label{eq:correlation_barriers}
    \mathcal F(t,\rho)=\sigma(\epsilon-R(t,\rho)),\quad
    \mathcal F^{\mathrm s}(t,\rho)=\epsilon_{\mathrm s}-S(t,\rho),
\end{equation}
where $\epsilon,\epsilon_{\mathrm s}>0$ and $\sigma\in\{-1,1\}$. For nonnegative kernels that are large for nearby states, $\sigma=1$ limits average proximity to an obstacle population, while $\sigma=-1$ maintains proximity to a target population. The self-correlation bound $S(t,\rho)\leq \epsilon^{\mathrm s}$ limits crowding within the swarm and can accompany either constraint.

\begin{proposition}\label{prop:cross}\label{prop:self}
Set $K_0=k$, $A_0=h$, and define
\begin{equation}\label{eq:kernel_recursions}
\begin{aligned}
    K_{j+1}&=\partial_tK_j+D_xK_j\, f(t,x)+D_yK_j\, f^*(t,y),\\
    A_{j+1}&=\partial_tA_j+D_xA_j\, f(t,x)+D_yA_j\, f(t,y).
\end{aligned}
\end{equation}
Then every $A_j$ is symmetric, and, for $j\geq0$,
\begin{equation}\label{eq:cross_derivatives}
\begin{aligned}
    \mathcal L_f^jR(t,\rho)&=\int_{\mathbb R^{2d}}K_j(t,x,y)\rho(x)\rho_t^*(y)\,dx\,dy,\\
    D^W_\rho [\mathcal L _f^jR(t,\rho)](x)&=\int_{\mathbb R^d}D_xK_j(t,x,y)\rho_t^*(y)\,dy,
\end{aligned}
\end{equation}
\begin{equation}\label{eq:self_derivatives}
\begin{aligned}
    \mathcal L_f^jS(\rho,t)&=\int_{\mathbb R^{2d}}A_j(t,x,y)\rho(x)\rho(y)\,dx\,dy,\\
    D^W_\rho [\mathcal L_f^jS(\rho,t)](x)&=2\int_{\mathbb R^d}D_xA_j(t,x,y)\rho(y)\,dy.
\end{aligned}
\end{equation}
\end{proposition}
\begin{proof}
The first variations follow from linearity in $\rho$ for $R_j$ and symmetry for $S_j$. Applying $\mathcal L_f$ and integrating by parts yields~\eqref{eq:kernel_recursions}: the second argument is transported by $f^*$ for cross-correlation and by $f$ for self-correlation. The latter recursion preserves symmetry, so induction gives the result.
\end{proof}

Hence, for barriers $\mathcal{F}$ and $\mathcal{F}^{\mathrm s}$ in~\eqref{eq:correlation_barriers} we obtain that
\begin{equation*}
    \mathcal L_f^j\mathcal F=-\sigma \mathcal L_f^j R,\qquad
    \mathcal L_f^j\mathcal F^{\mathrm s}=-\mathcal L_f^j S,\quad \forall j\geq 1.
\end{equation*}

\subsection{Position constraints for double integrators}\label{sec:DI}

Here, we discuss a specific example of functionals with higher relative degree. More specifically, we consider a swarm with the double integrator dynamics and \emph{position-only} correlation functionals. Hence, we aim at controlling the swarm through a distributed acceleration so that the swarm tracks or avoids a reference in the physical space.

To this end, let $x=(q,v)\in\mathbb R^{2n}$ be the state representing the position and velocity, and consider the dynamics
\begin{equation}\label{eq:DI}
    \dot q=v,~ \dot v=u\quad \Longleftrightarrow \quad
    f=\begin{pmatrix}v\\0\end{pmatrix},~
    g=\begin{pmatrix}0\\I_n\end{pmatrix}~ \text{in \eqref{eq:dynamics}}.
\end{equation}

Furthermore, let $k(t,x,y)=\kappa(t,q,p)$ and $h(t,x,y)=\kappa_{\mathrm s}(t,q,p)$, where $y=(p,w)$. Integrating out velocities in~\eqref{eq:cross}--\eqref{eq:self} shows that $R$ and $S$ depend only on the position marginals $\nu(q)=\int_{\mathbb R^n}\rho(q,v)\,dv$ and $\nu_t^*(p)=\int_{\mathbb R^n}\rho_t^*(p,w)\,dw$. Their derivatives, however, involve velocities.

The following proposition gives the drift derivatives and control terms needed to apply the higher-order CBF construction to these position constraints.

\begin{proposition}\label{prop:DI}
Assume that the functionals $\mathcal{F},\mathcal{F}^{\mathrm s}$ are as in~\eqref{eq:correlation_barriers}, and the reference density $\rho^*_t$ evolves under the prescribed dynamics $\dot p=w$, $\dot w=a^*(t,p,w)$, independently of the swarm control. Set $K_0=\kappa(t,q,p)$ and $A_0=\kappa_{\mathrm s}(t,q,p)$, and define
\begin{equation}\label{eq:DI_recursions}
\begin{aligned}
 K_{j+1}&=(\partial_t+v\cdot\nabla_q+w\cdot\nabla_p+a^*\cdot\nabla_w)K_j,\\
 A_{j+1}&=(\partial_t+v\cdot\nabla_q+w\cdot\nabla_p)A_j.
\end{aligned}
\end{equation}
Then, for $j\geq1$, we have that
\begin{equation}\label{eq:DI_drift_derivatives}
\begin{aligned}
 \mathcal L_f^j\mathcal F(t,\rho)
 &=-\sigma\int_{\mathbb R^{4n}}K_j(t,x,y)\rho(x)\rho_t^*(y)\,dx\,dy,\\
 \mathcal L_f^j\mathcal F^{\mathrm s}(t,\rho)
 &=-\int_{\mathbb R^{4n}}A_j(t,x,y)\rho(x)\rho(y)\,dx\,dy,
\end{aligned}
\end{equation}
where $x=(q,v)$, $y=(p,w)$. Thus, $\mathcal F$ and $\mathcal F^{\mathrm s}$ have relative degree $r=2$, and the control-affine terms in~\eqref{eq:control_general} are
\begin{equation}\label{eq:DI_control_terms}
\begin{aligned}
&\int_{\mathbb R^d}D^W_\rho[\mathcal L_f\mathcal F](t,\rho)\,g\, u\,\rho_t\,dx\\
&=-\sigma\int_{\mathbb R^{2n}}b(t,q)\cdot u(t,q,v)\rho(q,v)\,dq\,dv,\\
&\int_{\mathbb R^d}D^W_\rho[\mathcal L_f\mathcal F^{\mathrm{s}}](t,\rho)\,g\, u\,\rho_t\,dx\\
    &=-\int_{\mathbb R^{2n}}c(t,\rho,q)\cdot u(t,q,v)\rho(q,v)\,dq\,dv,
\end{aligned}
\end{equation}
where
\begin{equation}\label{eq:bc}
\begin{aligned}
    b(t,q)&=\int_{\mathbb R^n}\nabla_q\kappa(t,q,p)\nu_t^*(p)\,dp,\\
    c(t,\rho,q)&=2\int_{\mathbb R^n}\nabla_q\kappa_{\mathrm s}(t,q,p)\nu(p)\,dp.
\end{aligned}
\end{equation}
\end{proposition}
\begin{proof}
Substituting the double-integrator drift and the reference dynamics into Proposition~\ref{prop:cross} gives~\eqref{eq:DI_recursions} and~\eqref{eq:DI_drift_derivatives}. Since $g$ selects velocity derivatives, $D_vK_0=D_vA_0=0$ makes the first control terms vanish. One application of the kernel recursions gives $D_vK_1=D_q\kappa$ and $D_vA_1=D_q\kappa_{\mathrm s}$. The first-variation formulas in Proposition~\ref{prop:cross}, followed by integration over velocities, yield the stated control terms.
\end{proof}

\begin{remark}
    For the double integrator dynamics, setting $a^*(t,p,w)=0$ does not generally make the reference $\rho^*_t$ stationary. Thus, it is more convenient to treat the stationary reference case separately. For that, one simply needs to replace $\rho^*_t$ by $\rho^*$ in Proposition~\ref{prop:DI}, and discard the $w\cdot\nabla_p+a^*\cdot\nabla_w$ term in recursion for $(K_j)$ in~\eqref{eq:DI_recursions}. Proposition~\ref{prop:DI} remains valid.
\end{remark}

\begin{figure*}[!t]
    \centering
    \includegraphics[width=0.96\textwidth]{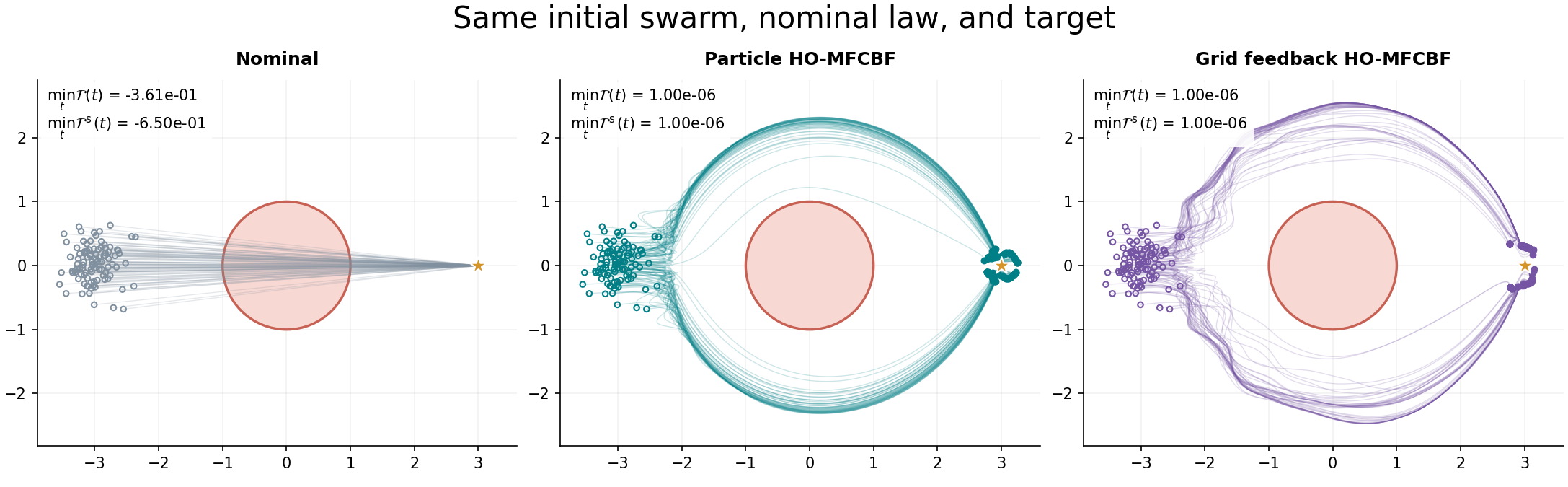}
    \caption{Obstacle avoidance: nominal, particle HO-MFCBF, and grid-feedback HO-MFCBF trajectories. The shaded disk is the obstacle and the star is the target. Annotations report the minimum empirical barrier values.}
    \label{fig:avoidance}
\end{figure*}

\begin{figure*}[!t]
    \centering
    \includegraphics[width=0.96\textwidth]{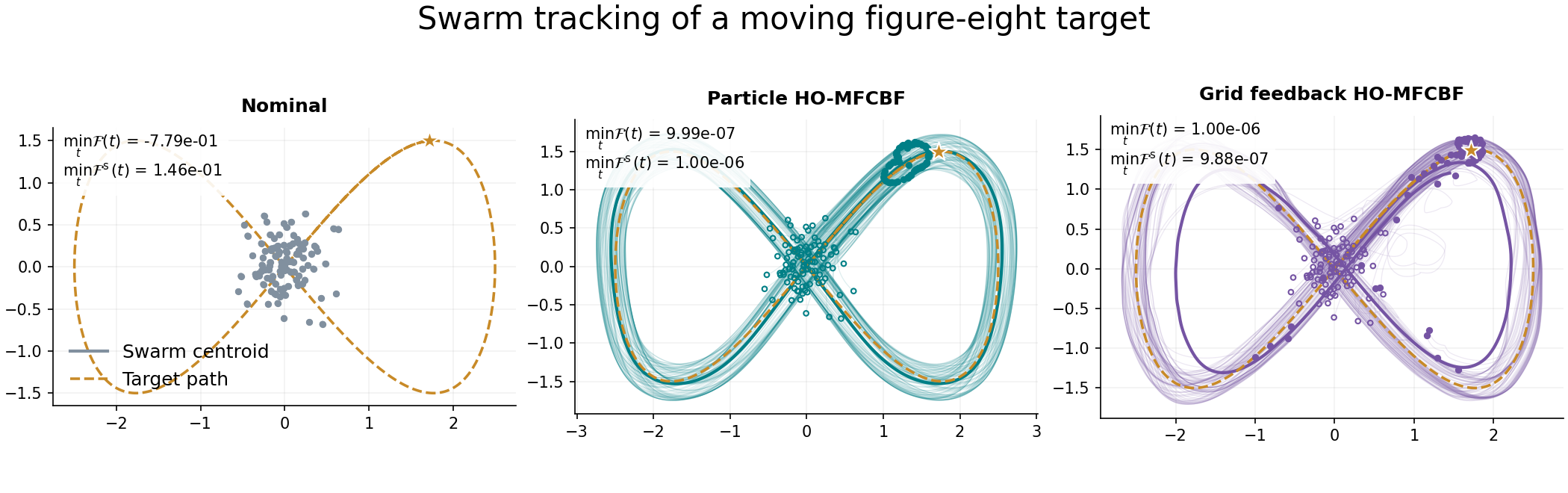}
    \caption{Moving-reference tracking with the same three controllers.
    The dashed curve is the target-center path and the star marks its final
    position. The nominal swarm remains stationary, whereas both filtered
    swarms follow the reference.}
    \label{fig:tracking}
\end{figure*}

\section{Numerical experiments}\label{sec:experiments}\label{sec:numerics}

We consider a swarm with a nominal control $u^{\mathrm{nom}}$. At each time, we find the closest safe control by solving
\begin{equation}\label{eq:continuous_projection}
    \underset{\substack{u\in\mathcal K_{\mathrm{CBF}}(t,\rho_t)\\
                       u\in\mathcal K^{\mathrm s}_{\mathrm{CBF}}(t,\rho_t)}}{\operatorname{argmin}}
    \frac12\big\|u-u^{\mathrm{nom}}(t,\cdot)\big\|^2,
\end{equation}
where $\mathcal K_{\mathrm{CBF}}$ and $\mathcal K^{\mathrm s}_{\mathrm{CBF}}$ are the safe-control sets in~\eqref{eq:control_general} for $\mathcal F$ and $\mathcal F^{\mathrm s}$, respectively. The norm can be chosen to specify how control modifications are measured.

We consider two discretizations of~\eqref{eq:continuous_projection}, where one chooses an acceleration for each particle, and the other constructs a feedback field by interpolation on a state-space grid.

\subsection{Discretization and safety filtering}

For simplicity, we choose $\alpha_i(s)=\lambda_i s$ and $\alpha_i^{\mathrm s}(s)=\mu_i s$, with $\lambda_i,\mu_i>0$, $i=1,2$. The drift contributions to the safe-control inequalities are
\begin{equation}\label{eq:betas}
\begin{aligned}
    \beta&=\lambda_1\lambda_2(\epsilon-R)-(\lambda_1+\lambda_2)\mathcal L_fR-\mathcal L_f^2R,\\
    \beta_{\mathrm s}&=\mu_1\mu_2(\epsilon_{\mathrm s}-S)-(\mu_1+\mu_2)\mathcal L_fS-\mathcal L_f^2S.
\end{aligned}
\end{equation}
We approximate the densities by empirical measures
\begin{equation}\label{eq:empirical}
    \rho^N=\frac1N\sum_{i=1}^N\delta_{x_i},\qquad
    \rho^{*,J}=\frac1J\sum_{a=1}^J\delta_{y_a},
\end{equation}
where $x_i=(q_i,v_i)$ and $y_a=(p_a,w_a)$. The kernel integrals in Proposition~\ref{prop:DI} become finite sums over swarm--reference pairs or pairs of swarm particles. These sums compute the drift derivatives in~\eqref{eq:betas}; likewise,~\eqref{eq:bc} gives $b_i=b(t,q_i)$ and $c_i=c(t,\rho^N,q_i)$. All coefficients below are evaluated at the current empirical densities.

\textit{Particle controls.} For this case, we use the $L^2(\rho_t)$ norm, approximated by the empirical swarm distribution. Thus, the safe accelerations solve
\begin{equation}\label{eq:particle_QP}
\begin{aligned}
    \min_{u_1,\ldots,u_N}\quad&\frac1{2N}\sum_{i=1}^N|u_i-u_i^{\mathrm{nom}}|^2\\
    \text{subject to}\quad&\frac\sigma N\sum_{i=1}^N b_i\cdot u_i\leq\sigma\beta,\\
    &\frac1N\sum_{i=1}^N c_i\cdot u_i\leq\beta_{\mathrm s}.
\end{aligned}
\end{equation}

\textit{Grid feedback.} For this case, we choose $L$ state-space nodes $\xi_a$ and interpolation functions $\phi_a$ satisfying $\phi_a(\xi_b)=\delta_{ab}$ and $\sum\limits_{a=1}^L\phi_a=1$. We then represent the control by
\begin{equation}\label{eq:grid_control}
    u(x)=\sum_{a=1}^L\phi_a(x)U_a,
\end{equation}
so the unknown parameters are now the coefficients $(U_a)$.

Furthermore, we use the $L^2$ norm with respect to the Lebesgue measure in~\eqref{eq:continuous_projection}, approximated by trapezoidal quadrature. The density is still represented by particles. Substituting the interpolated control into the same empirical constraints gives
\begin{equation}\label{eq:grid_QP}
\begin{aligned}
    \min_{U_1,\ldots,U_L}\quad&\frac12\sum_{a=1}^L\omega_a|U_a-U_a^{\mathrm{nom}}|^2\\
    \text{subject to}\quad&\sum_{a=1}^L \left[\frac\sigma N\sum_{i=1}^N b_i \phi_a(x_i)\right]U_a\leq\sigma\beta,\\
    &\sum_{a=1}^L \left[\frac1N\sum_{i=1}^N c_i \phi_a(x_i)\right] U_a\leq\beta_{\mathrm s},
\end{aligned}
\end{equation}
with the trapezoidal weights $\omega_a$. Interpolation restricts the control fields, and the nodal objective differs from the particle objective, so the two projections can produce different trajectories.

We impose no additional constraints on the accelerations. With the two affine barrier inequalities, both least-squares projections admit analytic solutions whenever feasible. Adding control bounds or other affine input constraints yields a quadratic program.

\subsection{Swarm avoidance and tracking}
We use the Gaussian position kernels
\begin{equation}\label{eq:Gaussian}
    \kappa(q,p)=e^{-|q-p|^2/(2\ell^2)},\qquad
    \kappa_{\mathrm s}(q,p)=e^{-|q-p|^2/(2\ell_{\mathrm s}^2)},
\end{equation}
where $\ell,\ell_{\mathrm s}>0$ set the proximity scales.

We simulate $N=100$ planar double integrators over $[0,24]$, comparing
nominal control with the particle and grid projections \eqref{eq:particle_QP} and~\eqref{eq:grid_QP}. Each comparison uses identical initial samples, zero velocities, and Gaussian position perturbations with covariance $0.25^2I$. Both filtered controllers enforce the cross- and self-correlation constraints separately, with  $\lambda_1=\lambda_2=\mu_1=\mu_2=2$,
$\ell_{\mathrm s}=0.18$, and $\epsilon_{\mathrm s}=0.35$. The feedback grid has $20\times20\times5\times5$ nodes in $(q,v)$, tensor-product piecewise-linear basis functions, and normalized trapezoidal weights. Its domain is $[-6,6]\times[-5,5]^3$ for avoidance and $[-4,4]\times[-3,3]^3$ for tracking. Fourth-order Runge--Kutta reevaluates the control at every stage, using time steps $0.02$ and $0.005$, respectively. 
To accommodate numerical integration and round-off errors, both experiments enforce the HO-MFCBF conditions on $\mathcal F-\delta$ and $\mathcal F^{\mathrm s}-\delta$, with $\delta=10^{-6}$, while reporting the original barrier values.

\textit{Obstacle avoidance.}
The swarm starts near $(-3,0)$ and moves toward $q_{\mathrm{tar}}=(3,0)$
under $u^{\mathrm{nom}}=0.7(q_{\mathrm{tar}}-q)-1.8v$. The reference is the uniform unit disk centered at the origin, represented
by $256$ equal-area quadrature points. We use $\sigma=1$, $\ell=0.45$, and $\epsilon=10^{-3}$. Figure~\ref{fig:avoidance} shows the nominal swarm crossing the obstacle and contracting toward the target. Both filtered runs retain positive clearance at recorded times, with the minimum distances to the disk boundary being 1e-6 for both particle and grid control.
Figure~\ref{fig:avoidance} also shows the self-correlation cap to be satisfied, which prevents complete collapse at the target.

\textit{Swarm tracking.}
The reference is a uniform disk of radius $0.25$, represented by $64$ quadrature points, with center
\begin{equation*}
\begin{aligned}
    c(t)&=(2.5\sin\theta(t),\,1.5\sin(2\theta(t))),\\
    \theta(t)&=0.32\big[t-2(1-e^{-t/2})\big].
\end{aligned}
\end{equation*}
The swarm starts near $c(0)=0$.
With $\sigma=-1$, $\ell=0.65$, and $\epsilon=0.78$, the cross-correlation
constraint maintains proximity to the moving reference. 
The nominal law $u^{\mathrm{nom}}=-1.8v$ leaves the initially resting swarm stationary, so its tracking barrier becomes negative. Both filtered swarms follow the figure-eight path (Fig.~\ref{fig:tracking}); grid feedback produces greater spatial spread, consistent with its different admissible fields and objective.

Their filtered minima are positive in avoidance and exceed 9e-7 in tracking, which ensures the swarm remains a certain distance from the target. The initial  conditions~\eqref{eq:initial_general} are satisfied, and 
the trajectories illustrate population-level constraint enforcement, without establishing continuous-time safety between samples or a minimum separation for every agent pair.

\section{Conclusion}\label{sec:conclusions}

We developed higher-order MF-CBFs to enforce safety constraints when control does not appear in the first derivative of mean-field barrier functionals. Our analytic framework enables efficient computation of higher-order derivatives and preserves structure in cross- and self-correlation functionals. We validated the method on double-integrator swarm tasks with $100$ agents. Code for these experiments can be found at \url{https://github.com/mines-opt-ml/higher-order-mf-cbf}. Future work includes analyzing global feasibility and fast computational methods for safe controls.

\section*{Acknowledgements}

OpenAI Codex and ChatGPT supported text editing, initial code development, and literature search. The mathematical results and final content are due to the authors.

\bibliographystyle{ieeetr}
\bibliography{refs}
\end{document}